\documentclass{amsart}
\usepackage{amsmath}
\usepackage{amsfonts}

\newtheorem{theorem}{Theorem}
\theoremstyle{plain}

\newtheorem{corollary}{Corollary}

\newtheorem{definition}{Definition}

\newtheorem{lemma}{Lemma}

\numberwithin{equation}{section}
\input{tcilatex}

\begin{document}
\title[On generalization of different type inequalities]{On generalization
of different type inequalities for $\left( \alpha ,m\right) $-convex
functions via fractional integrals}
\author{\.{I}mdat \.{I}\c{s}can}
\address{Department of Mathematics, Faculty of Sciences and Arts, Giresun
University, Giresun, Turkey}
\email{imdat.iscan@giresun.edu.tr}
\subjclass[2000]{ 26A51, 26A33, 26D10. }
\keywords{Hermite--Hadamard inequality, Riemann--Liouville fractional
integral, Ostrowski inequality, Simpson type inequality, $\left( \alpha
,m\right) $-convex function.}

\begin{abstract}
In this paper, new identity for fractional integrals have been defined. By
using of this identity, we obtained new general inequalities containing all
of Hadamard, Ostrowski and Simpson type inequalities for for functions whose
derivatives in absolute value at certain power are $\left( \alpha ,m\right) $%
-convex via Riemann Liouville fractional integral.
\end{abstract}

\maketitle

\section{Introduction}

Following inequalities are well known in the literature as Hermite-Hadamard
inequality, Ostrowski inequality and Simpson inequality respectively:

\begin{theorem}
Let $f:I\subseteq \mathbb{R\rightarrow R}$ be a convex function defined on
the interval $I$ of real numbers and $a,b\in I$ with $a<b$. The following
double inequality holds%
\begin{equation}
f\left( \frac{a+b}{2}\right) \leq \frac{1}{b-a}\dint\limits_{a}^{b}f(x)dx%
\leq \frac{f(a)+f(b)}{2}\text{.}  \label{1-1}
\end{equation}
\end{theorem}

\begin{theorem}
Let $f:I\subseteq \mathbb{R\rightarrow R}$ be a mapping differentiable in $%
I^{\circ },$ the interior of I, and let $a,b\in I^{\circ }$ with $a<b.$ If $%
\left\vert f^{\prime }(x)\right\vert \leq M$ for all $x\in \left[ a,b\right]
,$ then the following inequality holds%
\begin{equation*}
\left\vert f(x)-\frac{1}{b-a}\dint\limits_{a}^{b}f(t)dt\right\vert \leq 
\frac{M}{b-a}\left[ \frac{\left( x-a\right) ^{2}+\left( b-x\right) ^{2}}{2}%
\right]
\end{equation*}%
for all $x\in \left[ a,b\right] .$
\end{theorem}

\begin{theorem}
Let $f:\left[ a,b\right] \mathbb{\rightarrow R}$ be a four times
continuously differentiable mapping on $\left( a,b\right) $ and $\left\Vert
f^{(4)}\right\Vert _{\infty }=\underset{x\in \left( a,b\right) }{\sup }%
\left\vert f^{(4)}(x)\right\vert <\infty .$ Then the following inequality
holds:%
\begin{equation*}
\left\vert \frac{1}{3}\left[ \frac{f(a)+f(b)}{2}+2f\left( \frac{a+b}{2}%
\right) \right] -\frac{1}{b-a}\dint\limits_{a}^{b}f(x)dx\right\vert \leq 
\frac{1}{2880}\left\Vert f^{(4)}\right\Vert _{\infty }\left( b-a\right) ^{4}.
\end{equation*}
\end{theorem}

In \cite{T85}, G. Toader considered the class of m-convexfunctions: another
intermediate between the usual convexity and starshaped convexity.

\begin{definition}
The function $f:[0,b]\rightarrow 
\mathbb{R}
,b>0,$ is said to be $m-$convex, where $m\in \lbrack 0,1]$, if we have%
\begin{equation*}
f(tx+m(1-t)y)\leq tf(x)+m(1-t)f(y)
\end{equation*}%
for all $x,y\in \lbrack 0,b]$ and $t\in \lbrack 0,1]$. We say that $f$ is $%
m- $concave if $(-f)$ is $m-$convex.
\end{definition}

The class of $(\alpha ,m)$-convex functions was first introduced In \cite%
{M93}, and it is defined as follows:

\begin{definition}
The function $f:\left[ 0,b\right] \mathbb{\rightarrow R}$, $b>0$, is said to
be $(\alpha ,m)$-convex where $(\alpha ,m)\in \left[ 0,1\right] ^{2}$, if we
have%
\begin{equation*}
f\left( tx+m(1-t)y\right) \leq t^{\alpha }f(x)+m(1-t^{\alpha })f(y)
\end{equation*}%
for all $x,y\in \left[ 0,b\right] $ and $t\in \left[ 0,1\right] $.
\end{definition}

It can be easily that for $(\alpha ,m)\in \left\{ (0,0),(\alpha
,0),(1,0),(1,m),(1,1),(\alpha ,1)\right\} $ one obtains the following
classes of functions: increasing, $\alpha $-starshaped, starshaped, $m$%
-convex, convex, $\alpha $-convex.

Denote by $K_{m}^{\alpha }(b)$ the set of all $(\alpha ,m)$-convex functions
on $\left[ 0,b\right] $ for which $f(0)\leq 0$.

We give some necessary definitions and mathematical preliminaries of
fractional calculus theory which are used throughout this paper.

\begin{definition}
Let $f\in L\left[ a,b\right] $. The Riemann-Liouville integrals $%
J_{a^{+}}^{\theta }f$ and $J_{b^{-}}^{\theta }f$ of oder $\theta >0$ with $%
a\geq 0$ are defined by

\begin{equation*}
J_{a^{+}}^{\theta }f(x)=\frac{1}{\Gamma (\theta )}\dint\limits_{a}^{x}\left(
x-t\right) ^{\theta -1}f(t)dt,\ x>a
\end{equation*}

and

\begin{equation*}
J_{b^{-}}^{\theta }f(x)=\frac{1}{\Gamma (\theta )}\dint\limits_{x}^{b}\left(
t-x\right) ^{\theta -1}f(t)dt,\ x<b
\end{equation*}%
respectively, where $\Gamma (\theta )$ is the Gamma function defined by $%
\Gamma (\theta )=$ $\dint\limits_{0}^{\infty }e^{-t}t^{\theta -1}dt$ and $%
J_{a^{+}}^{0}f(x)=J_{b^{-}}^{0}f(x)=f(x).$
\end{definition}

In the case of $\theta =1$, the fractional integral reduces to the classical
integral. Properties concerning this operator can be found \cite%
{GM97,MR93,P99}.

In recent years, many athors have studied errors estimations for
Hermite-Hadamard, Ostrowski and Simpson inequalities. For recent some
results and generalizations concerning $(\alpha ,m)$-convex functions see 
\cite{BOP08,I13,I13b,I13c,OAK11,OKS10,OSS11,SA11,SO12,SSYB11,P12}.

In this paper, new identity for fractional integrals have been defined. By
using of this identity, we obtained a generalization of Hadamard, Ostrowski
and Simpson type inequalities for $\left( \alpha ,m\right) $-convex
functions via Riemann Liouville fractional integral.

\section{\ Main Results}

Let $f:I\subseteq 
\mathbb{R}
\rightarrow 
\mathbb{R}
$ be a differentiable function on $I^{\circ }$, the interior of $I$,
throughout this section we will take%
\begin{eqnarray*}
&&S_{f}\left( mx,\lambda ,\theta ,ma,mb\right) \\
&=&\left( 1-\lambda \right) m^{\theta -1}\left[ \frac{\left( x-a\right)
^{\theta }+\left( b-x\right) ^{\theta }}{b-a}\right] f(mx)+\lambda m^{\theta
-1}\left[ \frac{\left( x-a\right) ^{\theta }f(ma)+\left( b-x\right) ^{\theta
}f(mb)}{b-a}\right] \\
&&-\frac{\Gamma \left( \theta +1\right) }{m(b-a)}\left[ J_{\left( mx\right)
^{-}}^{\theta }f(ma)+J_{\left( mx\right) ^{+}}^{\theta }f(mb)\right]
\end{eqnarray*}%
where $m\in \left( 0,1\right] ,$ $ma,b\in I$ with $a<b$, $\ x\in \lbrack
a,b] $ , $\lambda \in \left[ 0,1\right] $, $\theta >0$ and $\Gamma $ is
Euler Gamma function. In order to prove our main results we need the
following identity.

\begin{lemma}
\label{2.1a},Let $f:I\subset \mathbb{R\rightarrow R}$ be a differentiable
mapping on $I^{\circ }$ such that $f^{\prime }\in L[ma,mb]$, where $m\in
\left( 0,1\right] $, $ma,b\in I$ with $a<b$. Then for all $x\in \lbrack a,b]$
, $\lambda \in \left[ 0,1\right] $ and $\theta >0$ we have:%
\begin{eqnarray*}
&&S_{f}\left( mx,\lambda ,\theta ,ma,mb\right) =\frac{m^{\theta }\left(
x-a\right) ^{\theta +1}}{b-a}\dint\limits_{0}^{1}\left( t^{\theta }-\lambda
\right) f^{\prime }\left( tmx+m\left( 1-t\right) a\right) dt \\
&&+\frac{m^{\theta }\left( b-x\right) ^{\theta +1}}{b-a}\dint\limits_{0}^{1}%
\left( \lambda -t^{\theta }\right) f^{\prime }\left( tmx+m\left( 1-t\right)
b\right) dt.
\end{eqnarray*}
\end{lemma}

A simple proof of the equality can be done by performing an integration by
parts in the integrals from the right side and changing the variable. The
details are left to the interested reader.

\begin{theorem}
\label{2.3}Let $f:$ $I\subset \lbrack 0,\infty )\rightarrow 
\mathbb{R}
$ be a differentiable function on $I^{\circ }$ such that $f^{\prime }\in
L[ma,mb]$, where $m\in \left( 0,1\right] $, $ma,b\in I$ $^{\circ }$ with $%
a<b $. If $|f^{\prime }|^{q}$ is $(\alpha ,m)$-convex on $[ma,b]$ for some
fixed $q\geq 1$, $x\in \lbrack a,b]$, $\lambda \in \left[ 0,1\right] $ and $%
\theta >0$ then the following inequality for fractional integrals holds%
\begin{eqnarray}
&&\left\vert S_{f}\left( mx,\lambda ,\theta ,ma,mb\right) \right\vert
\label{2-4} \\
&\leq &\frac{m^{\theta }A_{1}^{1-\frac{1}{q}}\left( \theta ,\lambda \right) 
}{b-a}\left\{ \left( x-a\right) ^{\theta +1}\left( \left\vert f^{\prime
}\left( mx\right) \right\vert ^{q}A_{2}\left( \alpha ,\theta ,\lambda
\right) +m\left\vert f^{\prime }\left( a\right) \right\vert ^{q}A_{3}\left(
\alpha ,\theta ,\lambda \right) \right) ^{\frac{1}{q}}\right.  \notag \\
&&\left. +\left( b-x\right) ^{\theta +1}\left( \left\vert f^{\prime }\left(
mx\right) \right\vert ^{q}A_{2}\left( \alpha ,\theta ,\lambda \right)
+m\left\vert f^{\prime }\left( b\right) \right\vert ^{q}A_{3}\left( \alpha
,\theta ,\lambda \right) \right) ^{\frac{1}{q}}\right\}  \notag
\end{eqnarray}%
where 
\begin{eqnarray*}
A_{1}\left( \theta ,\lambda \right) &=&\frac{2\theta \lambda ^{1+\frac{1}{%
\theta }}+1}{\theta +1}-\lambda , \\
A_{2}\left( \alpha ,\theta ,\lambda \right) &=&\frac{2\theta \lambda ^{1+%
\frac{1+\alpha }{\theta }}}{\left( \alpha +1\right) \left( \alpha +\theta
+1\right) }+\frac{1}{\alpha +\theta +1}-\frac{\lambda }{\alpha +1},
\end{eqnarray*}%
\begin{eqnarray*}
A_{3}\left( \alpha ,\theta ,\lambda \right) &=&A_{1}\left( \theta ,\lambda
\right) -A_{2}\left( \alpha ,\theta ,\lambda \right) \\
&=&\frac{2\theta \lambda ^{1+\frac{1}{\theta }}+1}{\theta +1}-\frac{2\theta
\lambda ^{1+\frac{1+\alpha }{\theta }}}{\left( \alpha +1\right) \left(
\alpha +\theta +1\right) }-\frac{1}{\alpha +\theta +1}-\frac{\alpha \lambda 
}{\alpha +1}.
\end{eqnarray*}
\end{theorem}

\begin{proof}
From Lemma \ref{2.1a}, property of the modulus and the power-mean inequality
we have%
\begin{eqnarray}
&&\left\vert S_{f}\left( mx,\lambda ,\theta ,ma,mb\right) \right\vert  \notag
\\
&\leq &\frac{m^{\theta }\left( x-a\right) ^{\theta +1}}{b-a}\left(
\dint\limits_{0}^{1}\left\vert t^{\theta }-\lambda \right\vert dt\right) ^{1-%
\frac{1}{q}}\left( \dint\limits_{0}^{1}\left\vert t^{\theta }-\lambda
\right\vert \left\vert f^{\prime }\left( tmx+m\left( 1-t\right) a\right)
\right\vert ^{q}dt\right) ^{\frac{1}{q}}  \notag \\
&&+\frac{m^{\theta }\left( b-x\right) ^{\theta +1}}{b-a}\left(
\dint\limits_{0}^{1}\left\vert t^{\theta }-\lambda \right\vert dt\right) ^{1-%
\frac{1}{q}}\left( \dint\limits_{0}^{1}\left\vert t^{\theta }-\lambda
\right\vert \left\vert f^{\prime }\left( tmx+m\left( 1-t\right) b\right)
\right\vert ^{q}dt\right) ^{\frac{1}{q}}.  \label{2-4a}
\end{eqnarray}%
Since $|f^{\prime }|^{q}$ is $(\alpha ,m)$-convex on $[ma,b],$ for all $t\in %
\left[ 0,1\right] $%
\begin{equation*}
\left\vert f^{\prime }\left( tmx+m\left( 1-t\right) a\right) \right\vert
^{q}\leq t^{\alpha }\left\vert f^{\prime }\left( mx\right) \right\vert
^{q}+m\left( 1-t^{\alpha }\right) \left\vert f^{\prime }\left( a\right)
\right\vert ^{q}
\end{equation*}%
and%
\begin{equation*}
\left\vert f^{\prime }\left( tmx+m\left( 1-t\right) b\right) \right\vert
^{q}\leq t^{\alpha }\left\vert f^{\prime }\left( mx\right) \right\vert
^{q}+m\left( 1-t^{\alpha }\right) \left\vert f^{\prime }\left( b\right)
\right\vert ^{q}.
\end{equation*}%
Hence by simple computation we get%
\begin{equation}
\dint\limits_{0}^{1}\left\vert t^{\theta }-\lambda \right\vert dt=\frac{%
2\theta \lambda ^{1+\frac{1}{\theta }}+1}{\theta +1}-\lambda  \label{2-4aa}
\end{equation}%
\begin{eqnarray}
&&\dint\limits_{0}^{1}\left\vert t^{\theta }-\lambda \right\vert \left\vert
f^{\prime }\left( tmx+m\left( 1-t\right) a\right) \right\vert ^{q}dt  \notag
\\
&\leq &\dint\limits_{0}^{1}\left\vert t^{\theta }-\lambda \right\vert \left(
t^{\alpha }\left\vert f^{\prime }\left( mx\right) \right\vert ^{q}+m\left(
1-t^{\alpha }\right) \left\vert f^{\prime }\left( a\right) \right\vert
^{q}\right) dt  \notag
\end{eqnarray}%
\begin{equation}
=\left\vert f^{\prime }\left( mx\right) \right\vert ^{q}A_{2}\left( \alpha
,\theta ,\lambda \right) +m\left\vert f^{\prime }\left( a\right) \right\vert
^{q}A_{3}\left( \alpha ,\theta ,\lambda \right) ,  \label{2-4b}
\end{equation}%
and similarly%
\begin{eqnarray}
&&\dint\limits_{0}^{1}\left\vert t^{\alpha }-\lambda \right\vert \left\vert
f^{\prime }\left( tmx+m\left( 1-t\right) b\right) \right\vert ^{q}dt  \notag
\\
&\leq &\dint\limits_{0}^{1}\left\vert t^{\alpha }-\lambda \right\vert
t\left\vert f^{\prime }\left( mx\right) \right\vert ^{q}+m\left( 1-t\right)
\left\vert f^{\prime }\left( b\right) \right\vert ^{q}dt  \notag
\end{eqnarray}%
\begin{equation}
=\left\vert f^{\prime }\left( mx\right) \right\vert ^{q}A_{2}\left( \alpha
,\theta ,\lambda \right) +m\left\vert f^{\prime }\left( b\right) \right\vert
^{q}A_{3}\left( \alpha ,\theta ,\lambda \right) .  \label{2-4c}
\end{equation}%
If we use (\ref{2-4aa}), (\ref{2-4b}) and (\ref{2-4c}) in (\ref{2-4a}), we
obtain the desired result. This completes the proof.
\end{proof}

\begin{corollary}
In Theorem \ref{2.3},

\begin{enumerate}
\item If we take $\theta =1,$ then inequality (\ref{2-4}) reduced to the
following inequality%
\begin{eqnarray*}
&&\left\vert S_{f}\left( mx,\lambda ,1,ma,mb\right) \right\vert \\
&=&\left\vert \left( 1-\lambda \right) f(mx)+\lambda \left[ \frac{\left(
x-a\right) f(ma)+\left( b-x\right) f(mb)}{b-a}\right] -\frac{1}{m(b-a)}%
\dint\limits_{ma}^{mb}f(t)dt\right\vert \\
&\leq &\frac{mA_{1}^{1-\frac{1}{q}}\left( 1,\lambda \right) }{b-a}\left\{
\left( x-a\right) ^{2}\left( \left\vert f^{\prime }\left( mx\right)
\right\vert ^{q}A_{2}\left( \alpha ,1,\lambda \right) +m\left\vert f^{\prime
}\left( a\right) \right\vert ^{q}A_{3}\left( \alpha ,1,\lambda \right)
\right) ^{\frac{1}{q}}\right. \\
&&\left. +\left( b-x\right) ^{2}\left( \left\vert f^{\prime }\left(
mx\right) \right\vert ^{q}A_{2}\left( \alpha ,1,\lambda \right) +m\left\vert
f^{\prime }\left( b\right) \right\vert ^{q}A_{3}\left( \alpha ,1,\lambda
\right) \right) ^{\frac{1}{q}}\right\} .
\end{eqnarray*}

\item If we take $q=1,$ then inequality (\ref{2-4}) reduced to the following
inequality%
\begin{eqnarray*}
&&\left\vert S_{f}\left( mx,\lambda ,\theta ,ma,mb\right) \right\vert \\
&\leq &\frac{m^{\theta }}{b-a}\left\{ \left( x-a\right) ^{\theta +1}\left(
\left\vert f^{\prime }\left( mx\right) \right\vert A_{2}\left( \alpha
,\theta ,\lambda \right) +m\left\vert f^{\prime }\left( a\right) \right\vert
A_{3}\left( \alpha ,\theta ,\lambda \right) \right) \right. \\
&&\left. +\left( b-x\right) ^{\theta +1}\left( \left\vert f^{\prime }\left(
mx\right) \right\vert A_{2}\left( \alpha ,\theta ,\lambda \right)
+m\left\vert f^{\prime }\left( b\right) \right\vert A_{3}\left( \alpha
,\theta ,\lambda \right) \right) \right\}
\end{eqnarray*}

\item If we take $x=\frac{a+b}{2},\ \lambda =\frac{1}{3},$then we get the
following Simpson type inequality via fractional integrals%
\begin{eqnarray*}
&&\left\vert \frac{2^{\theta -1}}{m^{\theta -1}\left( b-a\right) ^{\theta -1}%
}S_{f}\left( m\left( \frac{a+b}{2}\right) ,\frac{1}{3},\theta ,ma,mb\right)
\right\vert \\
&=&\left\vert \frac{1}{6}\left[ f(ma)+4f\left( \frac{m\left( a+b\right) }{2}%
\right) +f(mb)\right] -\frac{\Gamma \left( \theta +1\right) 2^{\theta -1}}{%
m^{\theta }\left( b-a\right) ^{\theta }}\left[ J_{\left( \frac{m\left(
a+b\right) }{2}\right) ^{-}}^{\theta }f(ma)+J_{\left( \frac{m\left(
a+b\right) }{2}\right) ^{+}}^{\theta }f(mb)\right] \right\vert \\
&\leq &\frac{m\left( b-a\right) }{4}A_{1}^{1-\frac{1}{q}}\left( \theta ,%
\frac{1}{3}\right) \left\{ \left( \left\vert f^{\prime }\left( \frac{m\left(
a+b\right) }{2}\right) \right\vert ^{q}A_{2}\left( \alpha ,\theta ,\frac{1}{3%
}\right) +m\left\vert f^{\prime }\left( a\right) \right\vert ^{q}A_{3}\left(
\alpha ,\theta ,\frac{1}{3}\right) \right) ^{\frac{1}{q}}\right. \\
&&\left. +\left( \left\vert f^{\prime }\left( \frac{m\left( a+b\right) }{2}%
\right) \right\vert ^{q}A_{2}\left( \alpha ,\theta ,\frac{1}{3}\right)
+m\left\vert f^{\prime }\left( b\right) \right\vert ^{q}A_{3}\left( \alpha
,\theta ,\frac{1}{3}\right) \right) ^{\frac{1}{q}}\right\} .
\end{eqnarray*}

\item If we take $x=\frac{a+b}{2},\ \lambda =0,$then we get the following
midpoint type inequality via fractional integrals%
\begin{eqnarray*}
&&\frac{2^{\theta -1}}{m^{\theta -1}\left( b-a\right) ^{\theta -1}}%
\left\vert S_{f}\left( m\left( \frac{a+b}{2}\right) ,0,\theta ,ma,mb\right)
\right\vert \\
&=&\left\vert f\left( \frac{m\left( a+b\right) }{2}\right) -\frac{\Gamma
\left( \theta +1\right) 2^{\theta -1}}{m^{\theta }\left( b-a\right) ^{\theta
}}\left[ J_{\left( \frac{m\left( a+b\right) }{2}\right) ^{-}}^{\theta
}f(ma)+J_{\left( \frac{m\left( a+b\right) }{2}\right) ^{+}}^{\theta }f(mb)%
\right] \right\vert \\
&\leq &\frac{m\left( b-a\right) }{4}\left( \frac{1}{\theta +1}\right) \left( 
\frac{1}{\alpha +\theta +1}\right) ^{\frac{1}{q}}\left\{ \left[ \left(
\theta +1\right) \left\vert f^{\prime }\left( \frac{m\left( a+b\right) }{2}%
\right) \right\vert ^{q}+\alpha m\left\vert f^{\prime }\left( a\right)
\right\vert ^{q}\right] ^{\frac{1}{q}}\right. \\
&&\left. +\left[ \left( \theta +1\right) \left\vert f^{\prime }\left( \frac{%
m\left( a+b\right) }{2}\right) \right\vert ^{q}+\alpha m\left\vert f^{\prime
}\left( b\right) \right\vert ^{q}\right] ^{\frac{1}{q}}\right\} .
\end{eqnarray*}

\item If we take$\ \lambda =1,$then we get the following generalized
trapezoid type inequality via fractional integrals%
\begin{eqnarray*}
&&m^{1-\theta }\left\vert S_{f}\left( mx,1,\alpha ,ma,mb\right) \right\vert
\\
&=&\left\vert \frac{\left( x-a\right) ^{\theta }f(ma)+\left( b-x\right)
^{\theta }f(mb)}{b-a}-\frac{\Gamma \left( \theta +1\right) }{m^{\theta
}\left( b-a\right) }\left[ J_{mx^{-}}^{\theta }f(ma)+J_{mx^{+}}^{\theta
}f(mb)\right] \right\vert \\
&\leq &\frac{m}{b-a}\left( \frac{\theta }{\theta +1}\right) ^{1-\frac{1}{q}%
}\left( \frac{1}{\left( \alpha +1\right) \left( \theta +1\right) \left(
\alpha +\theta +1\right) }\right) ^{\frac{1}{q}} \\
&&\times \left\{ \left( x-a\right) ^{\theta +1}\left[ \theta \left( \alpha
+1\right) \left\vert f^{\prime }\left( mx\right) \right\vert ^{q}+m\alpha
\left( \alpha +\theta +2\right) \left\vert f^{\prime }\left( a\right)
\right\vert \right] ^{\frac{1}{q}}\right. \\
&&\left. +\left( b-x\right) ^{\theta +1}\left[ \theta \left( \alpha
+1\right) \left\vert f^{\prime }\left( mx\right) \right\vert ^{q}+m\alpha
\left( \alpha +\theta +2\right) \left\vert f^{\prime }\left( b\right)
\right\vert \right] ^{\frac{1}{q}}\right\} .
\end{eqnarray*}

\item If $\ \left\vert f^{\prime }(u)\right\vert \leq M$ for all $u\in \left[
ma,b\right] $ and $\lambda =0,$ then we get the following Ostrowski type
inequality via fractional integrals%
\begin{eqnarray*}
&&\left\vert \left[ \frac{\left( x-a\right) ^{\theta }+\left( b-x\right)
^{\theta }}{b-a}\right] f(mx)-\frac{\Gamma \left( \theta +1\right) }{%
m^{\theta }\left( b-a\right) }\left[ J_{mx^{-}}^{\theta
}f(ma)+J_{mx^{+}}^{\theta }f(mb)\right] \right\vert \\
&\leq &mM\left( \frac{1}{\theta +1}\right) \left( \frac{\alpha m+\theta +1}{%
\alpha +\theta +1}\right) ^{\frac{1}{q}}\left[ \frac{\left( x-a\right)
^{\theta +1}+\left( b-x\right) ^{\theta +1}}{b-a}\right] ,
\end{eqnarray*}%
for each $x\in \left[ a,b\right] .$
\end{enumerate}
\end{corollary}

\begin{theorem}
\label{2.4}Let $f:$ $I\subset \lbrack 0,\infty )\rightarrow 
\mathbb{R}
$ be a differentiable function on $I^{\circ }$ such that $f^{\prime }\in
L[ma,mb]$, where $m\in \left( 0,1\right] $, $ma,b\in I$ $^{\circ }$ with $%
a<b $. If $|f^{\prime }|^{q}$ is $(\alpha ,m)$-convex on $[ma,b]$ for some
fixed $q>1$, $x\in \lbrack a,b]$, $\lambda \in \left[ 0,1\right] $ and $%
\theta >0$ then the following inequality for fractional integrals holds%
\begin{eqnarray}
&&\left\vert S_{f}\left( mx,\lambda ,\theta ,ma,mb\right) \right\vert
\label{3-1} \\
&\leq &\frac{m^{\theta }A_{4}\left( \theta ,\lambda ,p\right) ^{\frac{1}{p}}%
}{b-a}\left\{ \left( x-a\right) ^{\theta +1}\left( \frac{\left\vert
f^{\prime }\left( mx\right) \right\vert ^{q}+\alpha m\left\vert f^{\prime
}\left( a\right) \right\vert ^{q}}{\alpha +1}\right) ^{\frac{1}{q}}\right. 
\notag \\
&&\left. +\left( b-x\right) ^{\theta +1}\left( \frac{\left\vert f^{\prime
}\left( mx\right) \right\vert ^{q}+\alpha m\left\vert f^{\prime }\left(
b\right) \right\vert ^{q}}{\alpha +1}\right) ^{\frac{1}{q}}\right\}  \notag
\end{eqnarray}%
where 
\begin{eqnarray*}
&&A_{4}\left( \theta ,\lambda ,p\right) \\
&=&\left\{ 
\begin{array}{cc}
\frac{1}{\theta p+1}, & \lambda =0 \\ 
\begin{array}{c}
\frac{\lambda ^{\frac{\theta p+1}{\theta }}}{\theta }\left\{ \beta \left( 
\frac{1}{\alpha },p+1\right) +\frac{\left( 1-\lambda \right) ^{p+1}}{p+1}%
\right. \\ 
\left. \times _{2}F_{1}\left( \frac{1}{\theta }+p+1,p+1,p+2;1-\lambda
\right) \right\}%
\end{array}%
, & 0<\lambda <1 \\ 
\frac{1}{\theta }\beta \left( p+1,\frac{1}{\theta }\right) , & \lambda =1%
\end{array}%
\right. ,
\end{eqnarray*}%
$\beta $ is Euler Beta function defined by%
\begin{equation*}
\beta \left( x,y\right) =\frac{\Gamma (x)\Gamma (y)}{\Gamma (x+y)}%
=\dint\limits_{0}^{1}t^{x-1}\left( 1-t\right) ^{y-1}dt,\ \ x,y>0,
\end{equation*}%
$_{2}F_{1}$ is hypergeometric function defined by 
\begin{equation*}
_{2}F_{1}\left( a,b;c;z\right) =\frac{1}{\beta \left( b,c-b\right) }%
\dint\limits_{0}^{1}t^{b-1}\left( 1-t\right) ^{c-b-1}\left( 1-zt\right)
^{-a}dt,\ c>b>0,\ \left\vert z\right\vert <1\text{ (see \cite{AS65}),}
\end{equation*}%
and $\frac{1}{p}+\frac{1}{q}=1.$
\end{theorem}

\begin{proof}
From Lemma \ref{2.1a}, property of the modulus and using the H\"{o}lder
inequality we have%
\begin{eqnarray}
&&\left\vert S_{f}\left( mx,\lambda ,\theta ,ma,mb\right) \right\vert  \notag
\\
&\leq &\frac{m^{\theta }\left( x-a\right) ^{\theta +1}}{b-a}\left(
\dint\limits_{0}^{1}\left\vert t^{\theta }-\lambda \right\vert ^{p}dt\right)
^{\frac{1}{p}}\left( \dint\limits_{0}^{1}\left\vert f^{\prime }\left(
tmx+m\left( 1-t\right) a\right) \right\vert ^{q}dt\right) ^{\frac{1}{q}} 
\notag \\
&&+\frac{m^{\theta }\left( b-x\right) ^{\theta +1}}{b-a}\left(
\dint\limits_{0}^{1}\left\vert t^{\theta }-\lambda \right\vert ^{p}dt\right)
^{\frac{1}{p}}\left( \dint\limits_{0}^{1}\left\vert f^{\prime }\left(
tmx+m\left( 1-t\right) b\right) \right\vert ^{q}dt\right) ^{\frac{1}{q}}.
\label{3-a}
\end{eqnarray}%
Since $|f^{\prime }|^{q}$ is $(\alpha ,m)$-convex on $[ma,b],$ we get%
\begin{eqnarray}
\dint\limits_{0}^{1}\left\vert f^{\prime }\left( tmx+m\left( 1-t\right)
a\right) \right\vert ^{q}dt &\leq &\dint\limits_{0}^{1}t^{\alpha }\left\vert
f^{\prime }\left( mx\right) \right\vert ^{q}+m\left( 1-t^{\alpha }\right)
\left\vert f^{\prime }\left( a\right) \right\vert ^{q}dt  \notag \\
&=&\frac{\left\vert f^{\prime }\left( mx\right) \right\vert ^{q}+\alpha
m\left\vert f^{\prime }\left( a\right) \right\vert ^{q}}{\alpha +1},
\label{3-b}
\end{eqnarray}%
\begin{eqnarray}
\dint\limits_{0}^{1}\left\vert f^{\prime }\left( tmx+m\left( 1-t\right)
b\right) \right\vert ^{q}dt &\leq &\dint\limits_{0}^{1}t^{\alpha }\left\vert
f^{\prime }\left( mx\right) \right\vert ^{q}+m\left( 1-t^{\alpha }\right)
\left\vert f^{\prime }\left( b\right) \right\vert ^{q}dt  \notag \\
&=&\frac{\left\vert f^{\prime }\left( mx\right) \right\vert ^{q}+\alpha
m\left\vert f^{\prime }\left( b\right) \right\vert ^{q}}{\alpha +1},
\label{3-c}
\end{eqnarray}%
and by simple computation%
\begin{eqnarray}
&&\dint\limits_{0}^{1}\left\vert t^{\theta }-\lambda \right\vert ^{p}dt
\label{3-d} \\
&=&\left\{ 
\begin{array}{cc}
\frac{1}{\theta p+1}, & \lambda =0 \\ 
\begin{array}{c}
\frac{\lambda ^{\frac{\theta p+1}{\theta }}}{\theta }\left\{ \beta \left( 
\frac{1}{\theta },p+1\right) +\frac{\left( 1-\lambda \right) ^{p+1}}{p+1}%
\right. \\ 
\left. \times _{2}F_{1}\left( \frac{1}{\theta }+p+1,p+1,2+p;1-\lambda
\right) \right\}%
\end{array}%
, & 0<\lambda <1 \\ 
\frac{1}{\theta }\beta \left( p+1,\frac{1}{\theta }\right) , & \lambda =1%
\end{array}%
\right.  \notag
\end{eqnarray}%
Hence, If we use (\ref{3-b})-(\ref{3-d}) in (\ref{3-a}), we obtain the
desired result. This completes the proof.
\end{proof}

\begin{corollary}
In Theorem \ref{2.4},

\begin{enumerate}
\item If we take $\theta =1,$ then inequality (\ref{2-4}) reduced to the
following inequality%
\begin{eqnarray*}
&&\left\vert \left( 1-\lambda \right) f(mx)+\lambda \left[ \frac{\left(
x-a\right) f(ma)+\left( b-x\right) f(mb)}{b-a}\right] -\frac{1}{m(b-a)}%
\dint\limits_{ma}^{mb}f(t)dt\right\vert \\
&\leq &\frac{mA_{4}\left( 1,\lambda ,p\right) ^{\frac{1}{p}}}{b-a}\left\{
\left( x-a\right) ^{2}\left( \frac{\left\vert f^{\prime }\left( mx\right)
\right\vert ^{q}+\alpha m\left\vert f^{\prime }\left( a\right) \right\vert
^{q}}{\alpha +1}\right) ^{\frac{1}{q}}\right. \\
&&\left. +\left( b-x\right) ^{2}\left( \frac{\left\vert f^{\prime }\left(
mx\right) \right\vert ^{q}+\alpha m\left\vert f^{\prime }\left( b\right)
\right\vert ^{q}}{\alpha +1}\right) ^{\frac{1}{q}}\right\} .
\end{eqnarray*}

\item If we take $x=\frac{a+b}{2},\ \lambda =\frac{1}{3},$then we get the
following Simpson type inequality via fractional integrals%
\begin{eqnarray*}
&&\left\vert \frac{1}{6}\left[ f(ma)+4f\left( \frac{m\left( a+b\right) }{2}%
\right) +f(mb)\right] -\frac{\Gamma \left( \theta +1\right) 2^{\theta -1}}{%
m^{\theta }\left( b-a\right) ^{\theta }}\left[ J_{\left( \frac{m\left(
a+b\right) }{2}\right) ^{-}}^{\theta }f(ma)+J_{\left( \frac{m\left(
a+b\right) }{2}\right) ^{+}}^{\theta }f(mb)\right] \right\vert \\
&\leq &\frac{m\left( b-a\right) }{4}A_{4}^{\frac{1}{p}}\left( \theta ,\frac{1%
}{3},p\right) \left\{ \left( \frac{\left\vert f^{\prime }\left( \frac{%
m\left( a+b\right) }{2}\right) \right\vert ^{q}+\alpha m\left\vert f^{\prime
}\left( a\right) \right\vert ^{q}}{\alpha +1}\right) ^{\frac{1}{q}}\right. \\
&&\left. +\left( \frac{\left\vert f^{\prime }\left( \frac{m\left( a+b\right) 
}{2}\right) \right\vert ^{q}+\alpha m\left\vert f^{\prime }\left( b\right)
\right\vert ^{q}}{\alpha +1}\right) ^{\frac{1}{q}}\right\} .
\end{eqnarray*}

\item If we take $x=\frac{a+b}{2},\ \lambda =0,$then we get the following
midpoint type inequality via fractional integrals%
\begin{eqnarray*}
&&\left\vert f\left( \frac{m\left( a+b\right) }{2}\right) -\frac{\Gamma
\left( \theta +1\right) 2^{\theta -1}}{m^{\theta }\left( b-a\right) ^{\theta
}}\left[ J_{\left( \frac{m\left( a+b\right) }{2}\right) ^{-}}^{\theta
}f(ma)+J_{\left( \frac{m\left( a+b\right) }{2}\right) ^{+}}^{\theta }f(mb)%
\right] \right\vert \\
&\leq &\frac{m\left( b-a\right) }{4}\left( \frac{1}{\theta p+1}\right) ^{%
\frac{1}{p}}\left\{ \left( \frac{\left\vert f^{\prime }\left( \frac{m\left(
a+b\right) }{2}\right) \right\vert ^{q}+\alpha m\left\vert f^{\prime }\left(
a\right) \right\vert ^{q}}{\alpha +1}\right) ^{\frac{1}{q}}\right. \\
&&\left. +\left( \frac{\left\vert f^{\prime }\left( \frac{m\left( a+b\right) 
}{2}\right) \right\vert ^{q}+\alpha m\left\vert f^{\prime }\left( b\right)
\right\vert ^{q}}{\alpha +1}\right) ^{\frac{1}{q}}\right\} .
\end{eqnarray*}

\item If we take$\ \lambda =1,$then we get the following generalized
trapezoid type inequality via fractional integrals%
\begin{eqnarray*}
&&\left\vert \frac{\left( x-a\right) ^{\theta }f(ma)+\left( b-x\right)
^{\theta }f(mb)}{b-a}-\frac{\Gamma \left( \theta +1\right) }{m^{\theta
}\left( b-a\right) }\left[ J_{mx^{-}}^{\theta }f(ma)+J_{mx^{+}}^{\theta
}f(mb)\right] \right\vert \\
&\leq &\frac{m\left( \frac{1}{\theta }\beta \left( p+1,\frac{1}{\theta }%
\right) \right) ^{\frac{1}{p}}}{b-a}\left\{ \left( x-a\right) ^{\theta
+1}\left( \frac{\left\vert f^{\prime }\left( mx\right) \right\vert
^{q}+\alpha m\left\vert f^{\prime }\left( a\right) \right\vert ^{q}}{\alpha
+1}\right) ^{\frac{1}{q}}\right. \\
&&\left. +\left( b-x\right) ^{\theta +1}\left( \frac{\left\vert f^{\prime
}\left( mx\right) \right\vert ^{q}+\alpha m\left\vert f^{\prime }\left(
b\right) \right\vert ^{q}}{\alpha +1}\right) ^{\frac{1}{q}}\right\}
\end{eqnarray*}

\item If $\ \left\vert f^{\prime }(u)\right\vert \leq M$ for all $u\in \left[
ma,b\right] $ and $\lambda =0,$ then we get the following Ostrowski type
inequality via fractional integrals%
\begin{eqnarray*}
&&\left\vert \left[ \frac{\left( x-a\right) ^{\theta }+\left( b-x\right)
^{\theta }}{b-a}\right] f(mx)-\frac{\Gamma \left( \theta +1\right) }{%
m^{\theta }\left( b-a\right) }\left[ J_{mx^{-}}^{\theta
}f(ma)+J_{mx^{+}}^{\theta }f(mb)\right] \right\vert \\
&\leq &mM\left( \frac{1}{\theta p+1}\right) ^{\frac{1}{p}}\left( \frac{%
1+\alpha m}{\alpha +1}\right) \left[ \frac{\left( x-a\right) ^{\theta
+1}+\left( b-x\right) ^{\theta +1}}{b-a}\right] ,
\end{eqnarray*}%
for each $x\in \left[ a,b\right] .$
\end{enumerate}
\end{corollary}

\end{document}